\documentclass{amsart}

\usepackage{amssymb}
\usepackage{amsthm}
\usepackage{amsmath}
\usepackage[all]{xy}

\theoremstyle{plain}
\newtheorem{thm}{Theorem}[section]
\newtheorem*{thm*}{Theorem}
\newtheorem{prop}[thm]{Proposition}
\newtheorem{lem}[thm]{Lemma}

\newtheorem{cor}[thm]{Corollary}
\newtheorem*{cor*}{Corollary}

\theoremstyle{definition}

\theoremstyle{remark}

\newcommand{\C}{\mathbb C}
\newcommand{\Q}{\mathbb Q}
\newcommand{\R}{\mathbb R}
\newcommand{\Z}{\mathbb Z}

\newcommand{\fraka}{\mathfrak{a}}
\newcommand{\frakd}{\mathfrak{d}}
\newcommand{\frakp}{\mathfrak{p}}

\newcommand{\la}{\langle}
\newcommand{\ra}{\rangle}

\newcommand{\eps}{\epsilon}

\DeclareMathOperator{\Gal}{Gal}
\DeclareMathOperator{\Spec}{Spec}
\DeclareMathOperator{\repart}{Re}

\newenvironment{enumalph}
{\begin{enumerate}}
{\end{enumerate}}

\newcommand{\sbet}{{}_{\textup{\'et}}}

\begin{document}

\title{The Gauss higher relative class number problem}
\author{John Voight}
\address{Department of Mathematics and Statistics, 16 Colchester Avenue, University of Vermont, Burlington, Vermont 05401-1455}
\email{jvoight@gmail.com}

\begin{abstract}
Assuming the $2$-adic Iwasawa main conjecture, we find all CM fields with higher relative class number at most $16$: there are at least $31$ and at most $34$ such fields, and at most one is nonabelian.
\end{abstract}

\maketitle

The problem of determining all imaginary quadratic fields $K=\Q(\sqrt{d})$ of small class number $h(K)$ was first posed in Article 303 of Gauss' \emph{Disquisitiones Arithmeticae}.  It would take almost 150 years of work, culminating in the results of Stark \cite{Stark} and Baker \cite{Baker}, to determine those fields with class number at most two: there are exactly $27$, the last having discriminant $d=-427$.  (See Goldfeld \cite{Goldfeld} or Stark \cite{StarkHistory} for a history of this problem.)  Significant further progress has been made recently by Watkins \cite{Watkins}, who enumerated all such fields $K$ with class number $h(K) \leq 100$.  

One interesting generalization of the Gauss class number problem is to replace $\Q$ by a totally real field $F$.  Let $K/F$ be a CM extension, i.e., $K$ is a totally imaginary quadratic extension of a totally real field $F$, and let $[F:\Q]=n$.  We have the divisibility relation $h(F) \mid h(K)$, and we denote by $h^-(K)=h(K)/h(F)$ the relative class number.  It is known that there are only finitely many CM fields with fixed relative class number \cite{StarkRelative}.  The complete list of CM fields of relative class number one is still unknown (see Lee-Kwon \cite{LeeKwon} for an overview).  However, many partial results are known: for example, there are exactly $302$ imaginary abelian number fields $K$ with relative class number one \cite{ChangKwon}, each having degree $[K:\Q] \leq 24$.  

The integer $h^-(K)$ can be determined by the analytic relative class number formula, as follows.  Let $\chi:\Gal(K/F) \to \{\pm 1\}$ denote the nontrivial character associated to the extension $K/F$ and let $L(\chi,s)$ denote the Artin $L$-function associated to $\chi$.  Then
\[ L(\chi,0)=\frac{2^n}{Q(K)}\frac{ h^-(K)}{w(K)},  \]
where $w(K)=\#\mu(K)$ is the number of roots of unity in the field $K$ and $Q(K)=[\Z_K^*:\Z_F^*\,\mu(K)] \in \{1,2\}$ is the Hasse $Q$-unit index.  

In this article, we consider the further generalization of the Gauss problem to higher relative class numbers of CM fields.   Let $E$ be a number field with ring of integers $\Z_E$ and let $m \in \Z_{\geq 3}$ be an odd integer.  We define the \emph{higher class group} of $E$ to be
\[ H^2(\Spec \Z_E,\Z(m)) = \prod_p H\sbet^2(\Spec \Z_E[1/p], \Z_p(m)). \]
The group $H^2(\Spec \Z_E,\Z(m))$ is finite, and we let $h_m(E)$ denote its order.  For $p \neq 2$, the Quillen-Lichtenbaum conjecture (which appears to have been proven by Voevodsky-Rust-Suslin-Weibel) implies that the \'etale $\ell$-adic Chern character
\[ K_{2m-2}(\Z_E) \otimes \Z_p \xrightarrow{\sim} H\sbet^2(\Spec \Z_E[1/p], \Z_p(m)) \]
is an isomorphism, thus $h_m(E)$ and $\#K_{2m-2}$ agree up to a power of two.  (See \S 1 for more detail.)

For the CM extension $K/F$, we define the \emph{higher relative class number} to be
\[ h_m^{-}(K) = \frac{h_m(K)}{h_m(F)} \in \Z. \]
In analogy with the usual class number, an analytic higher relative class number formula holds: up to a power of $2$, we have the equality
\begin{equation} \label{Lvalueconj}
L(\chi,1-m)=(-1)^{n\left(\frac{m-1}{2}\right)} \frac{2^{n+1}}{Q_m(K)} \frac{h_m^{-}(K)}{w_m(K)},
\end{equation}
where $w_m(K) \in \Z_{>0}$ and $Q_m(K) \in \{1,2\}$ are the number of higher roots of unity and the higher $Q$-index, respectively.  Assuming the Iwasawa main conjecture for $p=2$, the formula (\ref{Lvalueconj}) holds exactly (see Kolster \cite{Kolster} and the discussion below).

Our main result is as follows.

\begin{thm*}
Suppose that the Iwasawa main conjecture holds for $p=2$.  Then there are at least $31$ and at most $34$ pairs $(K/F,m)$ where $h_m^-(K) \leq 16$.
\end{thm*}

The extensions are listed in Tables 4.1--4.2 in Section 4, and explains the calculations of Henderson ($h_m^-(K)=1$ and $F=\Q$) reported in Kolster \cite[\S 3]{Kolster}.

We begin in Section 1 by giving the necessary background.  In Section 2, we estimate the size of the higher relative class number using the analytic formula in an elementary way, and prove a statement in the spirit of the Brauer-Siegel theorem.  In Section 3, we combine these estimates with the Odlyzko bounds to reduce the problem to a finite computation and then carry it out to prove our main result.  We conclude in Section 4 by tabulating the fields.

The author wishes to thank Manfred Kolster for suggesting this problem and for his helpful comments and the anonymous referee for a careful reading.

\section{Background}

In this section, we state in detail the higher relative class number formula, which determines analytically the order of the higher relative class number of a CM extension of number fields.  

We begin by introducing the $L$-function and its functional equation---see Lang \cite[Chapter XIII]{Lang} or Neukirch \cite[Chapter VII]{Neukirch} for a reference.  Let $E$ be a number field with ring of integers $\Z_E$, absolute discriminant $d_E$, and degree $[E:\Q]=r_1+2r_2=n$, where $r_1,r_2$ denote the number of real and pairs of complex places of $E$.  Let $\zeta_E(s)$ be the Dedekind zeta function of $E$, defined by
\[ \zeta_E(s)=\sum_{\fraka \subset \Z_E} \frac{1}{(N\fraka)^s}=\prod_{\frakp \subset \Z_E} \left(1-\frac{1}{(N\frakp)^s}\right)^{-1} \]
for $s \in \C$ with $\repart s > 1$.  Then $\zeta_E(s)$ has an analytic continuation to $\C \setminus \{1\}$ with a simple pole at $s=1$.  Define the completed zeta function of $E$ by
\[ \xi_E(s)=\left(\frac{d_E}{4^{r_2}\pi^n}\right)^{s/2} \Gamma(s/2)^{r_1} \Gamma(s)^{r_2} \zeta_E(s). \]
Then $\xi_E$ satisfies the functional equation $\xi_E(1-s)=\xi_E(s)$, and it follows that
\begin{equation} \label{funceq}
\zeta_E(1-s)=\zeta_E(s)\left(\frac{d_E}{4^{r_2} \pi^n}\right)^{s-1/2}\frac{\Gamma(s/2)^{r_1}\Gamma(s)^{r_2}}{\Gamma((1-s)/2)^{r_1}\Gamma(1-s)^{r_2}}.
\end{equation}

Now let $K/F$ be a CM extension of number fields with $[F:\Q]=n$ and let $\chi$ denote the nontrivial character of $\Gal(K/F)$.  Then the Artin L-function
\begin{equation} \label{defofL}
L(\chi,s) = \frac{\zeta_K(s)}{\zeta_F(s)}
\end{equation}
has an analytic continuation to $\C$, and for $s \in \C$ with $\repart s > 1$ we have
\begin{equation} \label{Lprod}
L(\chi,s) =\sum_{\fraka \subset \Z_F} \frac{\chi(\fraka)}{(N\fraka)^s} = \prod_{\frakp \subset \Z_F} \left(1-\frac{\chi(\frakp)}{(N\frakp)^s}\right)^{-1}.
\end{equation}

Applying equation (\ref{funceq}) to (\ref{defofL}) we obtain after simplification that
\begin{equation} \label{expandlval}
L(\chi,1-s)=L(\chi,s)\left(\frac{1}{(4\pi)^n}\frac{d_K}{d_F}\right)^{s-1/2} 
\left(\frac{\Gamma(s)\Gamma((1-s)/2)}{\Gamma(s/2)\Gamma(1-s)}\right)^n.
\end{equation}

We now define the higher relative class group, a group whose order is determined by values of $L(\chi,s)$ at negative even integers---see Kolster \cite{KolsterIntro,Kolster} for a more complete treatment.  Let $m \in \Z_{\geq 3}$ be odd.  For a prime $p$, we denote by $H\sbet^i(\Z_E[1/p], \Z_p(m))$ the $i$th \'etale cohomology group of $\Spec \Z_E[1/p]$ with $m$-fold twisted $\Z_p$-coefficients.  Define the ($m$th) \emph{higher class group} of $E$ by
\[ H^2(\Z_E,\Z(m)) = \prod_p H\sbet^2(\Z_E[1/p], \Z_p(m)). \]
There exists a homomorphism $K_{2m-2}(\Z_E) \to H^2(\Z_E,\Z(m))$ with finite cokernel \cite{DF, Soule} (in fact, supported at $2$), and it follows that $H^2(\Z_E,\Z(m))$ is finite, since $K_{2m-2}(\Z_E)$ is finite (a result of Borel \cite{Borel} and Quillen \cite{Quillen}).  We let $h_m(E)=\# H^2(\Z_E,\Z(m))$ denote the ($m$th) \emph{higher class number} of $E$.  We again have the divisibility $h_m(F) \mid h_m(K)$, and we define the \emph{higher relative class number} of the CM extension $K/F$ to be the quotient
\[ h_m^{-}(K) = \frac{h_m(K)}{h_m(F)} \in \Z. \]

Two other quantities appear in the higher class number formula.  First, let
\[ H^0(E,\Q/\Z(m))=\prod_p H^0(E, \Q_p/\Z_p(m)) \] 
be the group of \emph{higher roots of unity}, given in terms of Galois cohomology (invariants), and let $w_m(E)=\#H^0(E,\Q/\Z(m))$ denote its order.  By definition, $w_m(E)$ is the largest integer $q$ such that $G=\Gal(E(\zeta_q)/E)$ has exponent dividing $m$, where $\zeta_q$ denotes a primitive $q$th root of unity.

We have the following lemma that characterizes $w_m(E)$.  If $q$ is the power of a prime, let $\Q(\zeta_q)^{(m)}$ denote the subfield of $\Q(\zeta_q)$ of index $\gcd(\phi(q),m)$.

\begin{lem} \label{qmidwm}
If $q$ is a power of a prime, then $q \mid w_m(E)$ if and only if $E$ contains $\Q(\zeta_q)^{(m)}$.
\end{lem}

\begin{proof}
Note that since $\Q(\zeta_q)$ is Galois, we have
\[ G=\Gal(E(\zeta_q)/E) \cong \Gal(\Q(\zeta_q)/(E \cap \Q(\zeta_q))) \subset \Gal(\Q(\zeta_q)/\Q) \cong (\Z/q\Z)^\times. \]
Since $m \geq 3$ is odd, if $q$ is even then $q \mid w_m(E)$ if and only if $\Q(\zeta_q) \subset E$, as claimed.  If $q$ is odd, then since $\Gal(\Q(\zeta_q)/\Q)$ is cyclic, we have $q \mid w_m(E)$ if and only if $G$ has order dividing $m$ if and only if $\Q(\zeta_q)^{(m)} \subset E$.
\end{proof}

Finally, we define the \emph{higher $Q$-unit index} of the CM extension $K/F$ by
\[ Q_m(K)=[H\sbet^1(K, \Z_2(m)) : H\sbet^1(F, \Z_2(m)) H^0(K, \Q_2/\Z_2(m))]. \]
Collecting results of Kolster \cite[\S 3]{Kolster}, we know the following facts about $Q_m$.

\begin{prop} \label{Qmtests}
\ 
\begin{enumalph}
\item $Q_m \in \{1,2\}$.
\item $Q_m=2$ if and only if $H^1(K/F, H\sbet^1(K,\Z_2(m)))=0$.  
\item If an odd prime of $F$ ramifies in $K$, then $Q_m=1$.
\item If no odd prime of $F$ ramifies in $K$, the field $F$ has only one prime lying above $2$, and $h(F)$ is odd, then $Q_m=2$.
\end{enumalph}
\end{prop}

With these definitions, we now state the higher relative class number formula.

\begin{prop}[{Kolster \cite{Kolster}}] \label{hrcnf}
We have
\begin{equation} \label{hrcnfform}
L(\chi,1-m)=(-1)^{n\left(\frac{m-1}{2}\right)} \frac{2^{n+1}}{Q_m(K)} \frac{h_m^{-}(K)}{w_m(K)}
\end{equation}
up to a power of $2$.  If the $2$-adic Iwasawa main conjecture holds, then \textup{(\ref{hrcnfform})} holds.
\end{prop}

The proof of the Iwasawa main conjecture for $p$ odd by Wiles \cite{Wiles} implies that the formula (\ref{hrcnfform}) holds up to a power of $2$, and a proof of the $2$-adic main conjecture would imply it exactly.  Indeed, Kolster \cite{Kolster} has proven that (\ref{hrcnfform}) holds in the case that $K$ is abelian over $\Q$ by applying Wiles' proof of the $2$-adic main conjecture when $K$ is abelian.

\section{Estimates}

In this section, we estimate the size of the relative class number.  Throughout, we assume the truth of the higher relative class number formula (Proposition \ref{hrcnf}).  We retain the notation from the previous section, suppressing the dependence on $K$ whenever possible; in particular, we recall that $m \in \Z_{\geq 3}$ is odd.

We begin by substituting $s=m$ into (\ref{expandlval}).  From standard $\Gamma$-function identities, letting $m=2k-1$ we have
\[ \Gamma(m/2)=\frac{(2k)!}{4^k k!}\sqrt{\pi} \]
and
\[ \frac{\Gamma((1-m)/2)}{\Gamma(1-m)}=\frac{\Gamma(-k)}{\Gamma(-2k)}=(-1)^k\frac{(2k)!}{k!}. \]
Putting these together, we obtain
\[ \frac{\Gamma(m)\Gamma((1-m)/2)}{\Gamma(m/2)\Gamma(1-m)}=(-1)^{(m-1)/2}\frac{2^m(m-1)!}{\sqrt{\pi}} \]
hence
\begin{equation} \label{lchi1-m} L(\chi,1-m)=(-1)^{n\left(\frac{m-1}{2}\right)}L(\chi,m)\left(\frac{d_K}{d_F}\right)^{m-1/2}\left(\frac{2(m-1)!}{(2\pi)^m}\right)^n.
\end{equation}

Now, by the higher relative class number formula, we have
\begin{equation} \label{hmest}
h_m^{-}=|L(\chi,1-m)|\frac{w_m Q_m}{2^{n+1}} \geq \frac{|L(\chi,1-m)|}{2^{n+1}}
\end{equation}
since $Q_m,w_m \in \Z_{\geq 1}$.  From (\ref{lchi1-m}) and (\ref{hmest}) we obtain
\begin{equation} \label{est1}
h_m^{-} \geq \frac{1}{2}L(\chi,m)\left(\frac{d_K}{d_F}\right)^{m-1/2}\left(\frac{(m-1)!}{(2\pi)^m}\right)^n.
\end{equation}

We now estimate the value of $L(\chi,m)$. 

\begin{lem} \label{Lchigamma}
We have $\gamma(m)^n \leq L(\chi,m) \leq \zeta(m)^n$, where
\[ \gamma(m)=\prod_p \left(1+\frac{1}{p^m}\right)^{-1}. \]  
\end{lem}

\begin{proof}
From (\ref{Lprod}) we obtain
\[ L(\chi,m) = \prod_\frakp \left(1-\frac{\chi(\frakp)}{(N\frakp)^m}\right)^{-1}. \]
Organizing the product by primes $p$, we have
\[ \left(1+\frac{1}{p^m}\right)^{-n} \leq \prod_\frakp \left(1-\frac{\chi(\frakp)}{(N\frakp)^m}\right)^{-1} \leq \left(1-\frac{1}{p^m}\right)^{-n} \]
hence
\[ \gamma(m)^n \leq L(\chi,m) \leq \zeta(m)^n \]
as claimed.
\end{proof}

We compute easily that $\gamma(3) \geq 0.8463$, $\gamma(5) \geq 0.9653$, and $\gamma(m) \geq 0.9917$ if $m \geq 7$; clearly $\gamma(m) \to 1$ as $m \to \infty$.  Applying Lemma \ref{Lchigamma} to (\ref{est1}) we obtain
\begin{equation} \label{est!}
h_m^{-} \geq \frac{1}{2} \left(\frac{d_K}{d_F}\right)^{m-1/2} \left(\gamma(m)\frac{(m-1)!}{(2\pi)^m}\right)^n.
\end{equation}

We pause to prove the following proposition, which is in the spirit of the Brauer-Siegel theorem.  We have $d_K/d_F=d_F N\frakd_{K/F} \geq d_F$, where $\frakd_{K/F} \subset \Z_F$ denotes the relative discriminant of $K/F$.  

\begin{prop}
Let $\{K_i/F_i\}_i$ be a sequence of CM extensions with 
\[ [F_i:\Q] = o\left(\log(d_{K_i}/d_{F_i})\right). \]
Then
\[ \log h_m^{-}(K_i) \sim (m-1/2)\log(d_{K_i}/d_{F_i}) \]
as $i \to \infty$.
\end{prop}

\begin{proof}
Let $K/F$ be a CM extension with $[F:\Q]=n$.  From (\ref{est!}) we have
\begin{equation} \label{hmlower}
\log h_m^- \geq (m-1/2)\log(d_K/d_F) + n c(m)
\end{equation}
where $c(m)$ is a constant depending only on $m$.  On the other hand, by Lemma \ref{Lchigamma} we have $L(\chi,m) \leq \zeta(m)^n$.  Then applying equations (\ref{lchi1-m}) and (\ref{hmest}) with this estimate we obtain
\begin{equation} \label{hmupper}
\log h_m^{-} \leq \log w_m + (m-1/2)\log(d_K/d_F) + n \log \zeta(m).
\end{equation}
Putting together (\ref{hmlower}) and (\ref{hmupper}), since $[F_i:\Q]=o(\log(d_{K_i}/d_{F_i}))$ as $i \to \infty$, the result follows if we show that
\begin{equation} \label{wmK}
\log w_m(K_i) = o( \log(d_{K_i}/d_{F_i}) ).
\end{equation}

To prove (\ref{wmK}), we compare $w_m(K)$ and $d_K/d_F$ for a CM extension $K/F$.  Let $q=p^r$ be a power of a prime and suppose that $q \mid w_m$.  By Lemma \ref{qmidwm}, the field $K$ must contain $\Q(\zeta_q)^{(m)}$.  Suppose first that $q$ is odd.  It follows that $F$ must contain the field $\Q(\zeta_q)^{(2m)}$; since $d_{\Q(\zeta_q)^{(2m)}} \mid d_F$, by the conductor-discriminant formula \cite[Theorem 3.11]{Washington} we have
\[ p^{\phi(q)/(2m)-1} \mid d_{\Q(\zeta_q)^{(2m)}}. \]
If $q$ is even, then $F$ must contain the totally real subfield $\Q(\zeta_q)^+$ of $\Q(\zeta_q)$, and we similarly conclude that $p^{\phi(q)/2-1} \mid d_F$.  Define the multiplicative function $f$ with the value $f(p^r)=p^{\phi(p^r)/(2m)-1}$ for a prime power $p^r$. 

Now let $K_i/F_i$ be a subsequence with $w_m(K_i) \to \infty$; if no such subsequence exists, then we are done.  We now show that for any sequence of positive integers $n_i \to \infty$, we have $\log n_i = o(\log f(n_i))$.  The result then follows as
\[ \log w_{m_i} = o(\log f(w_{m_i})) = o(\log d_{F_i}) = o(\log(d_{K_i}/d_{F_i})) \]
as desired.

So let $\eps > 0$.  We prove that 
\[ r < \eps \left(\frac{\phi(p^r)}{2m}-1\right) \]
for all sufficiently large prime powers $p^r$, or also it is sufficient to show that
\begin{equation} \label{stupidpr}
\frac{\phi(p^r)}{4rm} > \frac{1}{\eps}.
\end{equation}
For $r=1$, clearly the inequality (\ref{stupidpr}) will be satisfied for $p$ sufficiently large; but then for the finitely many remaining primes $p$, the inequality holds for a sufficiently large power $r$.  Therefore it holds for any sufficiently large prime power $p^r$, and claimed.
\end{proof}

\begin{cor}
Let $F$ be a totally real field and $m \in \Z_{\geq 3}$ be odd.  Then for any $h \in \Z_{\geq 1}$, there are only finitely many CM extensions $K/F$ with $h_m^-(K) \leq h$.
\end{cor}

We return to estimating $h_m^{-}$ from below.  Rewriting (\ref{est!}) we obtain
\begin{equation} \label{est!!}
d_F \leq \frac{d_K}{d_F} \leq (2h_m^-)^{1/(m-1/2)} C(m)^n 
\end{equation}
where
\[ C(m)=\left(\frac{(2\pi)^m}{\gamma(m)(m-1)!}\right)^{1/(m-1/2)} \]  
depends only on $m$.  We compute that
\[ C(3) \leq 7.3517, \quad C(5) \leq 3.8332, \quad C(7) \leq 2.6336, \quad C(9) \leq 2.011. \]
It follows from Stirling's approximation and the fact that $\gamma(m)$ is increasing to $1$ that $C(m)$ is decreasing to $0$.

Let $\delta_F=d_F^{1/n}$ denote the root discriminant of $F$.   Taking $n$th roots we obtain from (\ref{est!!}) that
\begin{equation} \label{mainest}
\delta_F \leq (2h_m^-)^{2/\left((2m-1)n\right)} C(m).
\end{equation}

\section{Enumerating the list of extensions}

We now apply the results of \S 2 to enumerate CM extensions with small higher relative class number.  

We list those extensions with $h_m^- \leq 16$; we have chosen this bound to capture the smallest higher relative class number of a nonabelian field, and we note that this bound can easily be increased, if desired.  Let $NF_m(n)$ denote the set of totally real fields $F$ of degree $n$ that satisfy the bound (\ref{mainest}) with $h_m^-=16$, and let $NF_m = \bigcup_n NF_m(n)$.  

For $m=3$, the estimate (\ref{mainest}) then reads
\begin{equation} \label{deltaFest}
\delta_F \leq 32^{2/(5n)} C(3) \leq 7.3517 \cdot 4^{1/n}.
\end{equation}

By the (unconditional) Odlyzko bounds \cite{Martinet}, if $n \geq 7$ we have $\delta_F \geq 9.301$, but by (\ref{deltaFest}) we have $\delta_F \leq 8.962$, a contradiction.  In Table 2.3, for each degree $n \geq 2$ we list the upper bound on the root discriminant $\delta_F$, the corresponding Odlyzko bound $B_O$, the corresponding upper bound on the discriminant $d_F$, and the size of $NF_3(n)$.

\begin{center}
\vspace{1ex}
\textbf{Table 2.3.} Degree and Root Discriminant Bounds \\
\vspace{-2ex}
\[ 
\begin{array}{c|cccccc}
n & 2 & 3 & 4 & 5 & 6 & \geq 7 \\
\hline
\delta_F & \leq 14.703 & 11.670 & 10.397 & 9.701 & 9.263 & 8.962 \\
B_O & > 2.223 & 3.610 & 5.067 & 6.523 & 7.941 & 9.301 \\
\hline
d_F & \leq 216 & 1589 & 11684 & 85899 & 631505 & - \\
\#NF_3(n) & 65 & 48 & 64 & 8 & 6 & 0
\end{array}
\]
\vspace{0ex}
\end{center}
The fields $NF_3(n)$ of such small discriminant are well-known \cite{Tables}.

For $m=5$, arguing in a similar way we have 
\[ \delta_F \leq 32^{2/(9n)} C(5) \leq 3.8332 \cdot 2.1602^{1/n}, \] 
which already for $n \geq 4$ gives $\delta_F \leq 4.5273$, contradicting the Odlyzko bound $\delta_F > 5.067$.  In fact, we have $d_F \leq 28$ for $n=2$ and $d_F \leq 109$ for $n=3$, so that $\#NF_5(2)=8$ and $\#NF_5(3)=2$.  Proceeding in this way, we find:
\[
NF_m=
\begin{cases}
\{\Q,\Q(\sqrt{5}),\Q(\sqrt{8})\}, & \text{ if $m=7$}; \\
\{\Q,\Q(\sqrt{5})\}, & \text{ if $m=9$}; \\
\{\Q\}, & \text{ if $11 \leq m \leq 19$}; \\
\emptyset, & \text{ if $m \geq 21$}. \\
\end{cases} \]

Now for each such field $F$ and $m$, we have from (\ref{est!!}) that
\begin{equation} \label{boundk}
d_K \leq \lfloor 32^{2/(2m-1)} C(m)^n \rfloor d_F,
\end{equation}
leaving only finitely many possibilities for the CM extension $K/F$.  We can find these relative quadratic extensions explicitly by using a relative version of Hunter's theorem due to Martinet: see Cohen \cite[\S\S 9.2--9.3]{Cohen} for more details.  We obtain in this way $90,9,2,1$ extensions $K/F$ for $m=3,5,7,9$, and none for $m \geq 11$.

Next, using (\ref{lchi1-m}) and the higher relative class number formula (Proposition \ref{hrcnf}) we numerically compute the value $h_m^-/(w_m Q_m) \in \R$.  In order to recover this value exactly, it suffices to bound the size of the denominator.  We have $Q_m \in \{1,2\}$.  

To determine $w_m$, we apply Lemma \ref{qmidwm}.  If $q$ is the power of a prime and $q \mid m$, then $q$ is odd and $K$ contains the unique subfield of index $m$ of $\Q(\zeta_q)$.  In particular, this implies that $\phi(q)/m \mid 2n=[K:\Q]$, which already gives a bound on $w_m$.  To reduce the size of this bound further, we note that we also have $d_{\Q(\zeta_q)^{(m)}} \mid d_K$, so in particular $q \mid d_K$ whenever $\phi(q) > m$; furthermore, for any prime $\frakp$ of $K$ that is prime to $qd_K$, it is easy to see that the order of $N\frakp \in (\Z/q\Z)^*$ must divide $m$.  A prime power $q$ that passes these tests, the latter for sufficiently many primes $\frakp$, is very likely to divide $w_m$.  To compute $w_m$ exactly and verify that indeed $q \mid w_m$, we simply check if the $q$th cyclotomic polynomial over $K$ factors into polynomials of degree at most $m$.

In this way, we compute $h_m^-/Q_m \in \frac{1}{2}\Z$.  If this value is not an integer, then we immediately know $Q_m=2$.  Otherwise, we may apply the tests of Proposition \ref{Qmtests} to determine the value $Q_m$ in almost all cases.  We are lucky that in many cases where we cannot determine the value of $Q_m$, we nevertheless have $h_m^- \geq h_m^-/Q_m > 16$.  We have $6$ remaining cases.  One of which we can resolve as follows: for $F=\Q(\sqrt{30})$ and $K=\Q(\sqrt{-15},\sqrt{-2})$ with $h_3^-/Q_3=12$, we apply the divisibility result $8=h_3^-(\Q(\sqrt{-15})) \mid h_3^-(K)$ \cite[Corollary 3.6]{Kolster}, which implies $Q_3=2$, and hence $h_3^-=24$.  We were unable to resolve, and we leave it an open problem to compute the higher $Q$-index $Q_m$ in the other $5$ cases.  We expect the problem to be nontrivial for the reason that already characterizing Hasse's $Q$-index is quite intricate (see Hasse \cite{Hasse}).

Of the $90,9,2,1$ CM extensions for $m=3,5,7,9$, respectively, $26-29,4,1,0$ have $h_m^- \leq 16$, and they are listed in Tables 4.1--4.2.

\section{Tables}

In this section, we present the tables of CM extensions with higher relative class number $h_m^- \leq 16$.  Below, we list the totally real field $F$ and its discriminant $d_F$, the CM field $K$, its absolute discriminant $d_K$ and the norm of the relative discriminant $N(\frakd_{K/F})$, an element $\delta \in F$ such that $K=F(\sqrt{\delta})$, and the higher class number $h_m^-$.  As usual, we let $\zeta_k$ denote a primitive $k$th root of unity, $\omega=\zeta_3$ and $i=\zeta_4$, and we define $\lambda_k = \zeta_k+1/\zeta_k$, so that $\Q(\lambda_k)=\Q(\zeta_k)^+$.

The computations were performed in \textsf{Magma} \cite{Magma} and \textsf{Sage} \cite{Sage}; total computing time was about $2$ minutes.  

It is interesting to note that there is no CM extension with higher relative class number $h_m^-=2$.  Also, note that there is a misprint in the computation of $h_m^-(\zeta_{12})=1$ in Kolster \cite{Kolster}.  

\begin{center}
\vspace{1ex}
\textbf{Table 4.1.} CM extensions with higher relative class number $h_m^- \leq 16$ for $m \geq 5$ \\
\vspace{-2ex}
\[
\begin{array}{cc|cccc|c}
d_F & F & d_K & K & N(\frakd_{K/F}) & \delta  & h_5^- \\ \hline \hline
\rule{0pt}{2.5ex} 
1 & \Q             & 3     & \Q(\omega) & 3  & -3 & 1 \\ \hline
\rule{0pt}{2.5ex} 
1 & \Q             & 4     & \Q(i) & 4  & -4 & 5 \\
12 & \Q(\sqrt{3})  & 144   & \Q(\zeta_{12}) & 1  & -1 & 5 \\ \hline
\rule{0pt}{2.5ex} 
1 & \Q             & 7     & \Q(\sqrt{-7}) & 7  & -7 & 16 \\
\hline \ \\ 
\rule{0pt}{2.5ex} d_F & F & d_K & K & N(\frakd_{K/F}) & a  & h_7^- \\ \hline \hline
\rule{0pt}{2.5ex} 1 & \Q             & 3     & \Q(\omega) & 3  & -3 & 7
\end{array}
\] 
\end{center}

\newpage

\begin{center}
\vspace{1ex}
\textbf{Table 4.2.} CM extensions with higher relative class number $h_3^- \leq 16$ \\
\vspace{-2ex}
\[
\begin{array}{cc|cccc|c}
d_F & F            & d_K   & K       & \hspace{-1ex}N(\frakd_{K/F}) & \delta  & h_3^- \\ \hline \hline
\rule{0pt}{2.5ex} 
1 & \Q             & 3     & \Q(\omega)           & 3  & -3 & 1 \\
1 & \Q             & 4     & \Q(i)                & 4  & -1 & 1 \\
5 & \Q(\sqrt{5})   & 125   & \Q(\zeta_5)          & 5  & \hspace{-3ex}-\sqrt{5}(1+\sqrt{5})/2 & 1 \\
12 & \Q(\sqrt{3})  & 144   & \Q(\zeta_{12})       & 1  & -1 & 1 \\ \hline
\rule{0pt}{2.5ex} 
1 & \Q             & 8     & \Q(\sqrt{-2})        & 8  & -2 & 3 \\
1 & \Q             & 11    & \Q(\sqrt{-11})       & 11 & -11 & 3 \\
8 & \Q(\sqrt{2})   & 256   & \Q(\zeta_8)          & 4  & -2 & 3 \\
24 & \Q(\sqrt{6})  & 576   & \Q(\sqrt{-2},\omega) & 1  & -2 & 3 \\
44 & \Q(\sqrt{11}) & 1936  & \Q(\sqrt{-11}, i)    & 1  & -1 & 3 \\ \hline
\rule{0pt}{2.5ex} 
33 & \Q(\sqrt{33}) & 1089  & \Q(\sqrt{-11},\omega)& 1  & -3 & 3\text{ or }6? \\ \hline
\rule{0pt}{2.5ex} 
60 & \Q(\sqrt{15}) & 3600  & \Q(\sqrt{-15}, i)    & 1  & -1 & 4\text{ or }8? \\ \hline
\rule{0pt}{2.5ex} 
1 & \Q             & 7     & \Q(\sqrt{-7})        & 7  & -7 & 8 \\
1 & \Q             & 15    & \Q(\sqrt{-15})       & 15 & -15 & 8 \\
5 & \Q(\sqrt{5})   & 225   & \Q(\sqrt{-15},\omega)& 9  & -3 & 8 \\
28 & \Q(\sqrt{7})  & 784   & \Q(\sqrt{-7},i)      & 1  & -1 & 8 \\
49 & \Q(\lambda_7) & 16807 & \Q(\zeta_7)          & 7  & -7 & 8 \\ \hline
\rule{0pt}{2.5ex} 
88 & \Q(\sqrt{22}) & 7744  & \Q(\sqrt{-11},\sqrt{-2})& 16 & -2 & 9 \\ \hline
\rule{0pt}{2.5ex} 
1 & \Q             & 19    & \Q(\sqrt{-19})       & 19 & -19 & 11 \\
57 & \Q(\sqrt{57}) & 3249  & \Q(\sqrt{-19},\omega)& 1  & -3 & 11\text{ or }22? \\
76 & \Q(\sqrt{19}) & 5776  & \Q(\sqrt{-19},i)     & 1  & -1 & 11 \\ \hline
\rule{0pt}{2.5ex}
% 120 & \Q(\sqrt{30})& 14400 & \Q(\sqrt{-15},\sqrt{-2})& 1 & -2 & 12\text{ or }24? \\ \hline
% \rule{0pt}{2.5ex}
81 & \Q(\lambda_9) & 19683 & \Q(\zeta_9)          & 3  & -3 & 13 \\ \hline
\rule{0pt}{2.5ex}
1  & \Q            & 20    & \Q(\sqrt{-5})        & 20 & -5 & 15 \\ 
5  & \Q(\sqrt{5})  & 400   & \Q(\sqrt{-5},i)       & 16 & -1 & 15 \\
60 & \Q(\sqrt{15}) & 3600  & \Q(\sqrt{-5},\omega) & 1  & -1 & 15\text{ or }30? \\
2000 & \Q(\lambda_{20}) & 4000000 & \Q(\zeta_{20})& 1  & -1 & 15 \\ \hline
\rule{0pt}{2.5ex}
21 & \Q(\sqrt{21}) & 441   & \Q(\sqrt{-7},\omega) & 1  & -3 & 16 \\
8 & \Q(\sqrt{2})   & 1088  & \Q\left(\sqrt{2\sqrt{2}-5}\right) & 17 & 2\sqrt{2}-5 & 16 \\
17 & \Q(\sqrt{17}) & 2312  & \textstyle{\Q\left(\sqrt{-(5+\sqrt{17})/2}\right)} & 8 & \hspace{-3ex}-(5+\sqrt{17})/2 & 16\text{ or }32? \\
1125 & \Q(\lambda_{15}) & 1265625 & \Q(\zeta_{15}) & 1 & -15 & 16
\end{array}
\] 
\vspace{0ex}
\end{center}

\end{document}